\documentclass[final,12pt]{elsarticle}
\usepackage{xcolor,soul,cancel}
\usepackage[color]{showkeys}
\usepackage{amssymb,amsmath,amsthm,mathtools}
\usepackage{graphicx,float,enumitem,hyperref}
\newlist{alphalist}{enumerate}{1} 
\setlist[alphalist]{label=(\alph*)} 
\journal{arXiv}

\definecolor{refkey}{rgb}{0,1,1}
\definecolor{labelkey}{rgb}{1,0,0}

\numberwithin{equation}{section}

\newtheorem{thm}{Theorem}[section]
\newtheorem{lem}[thm]{Lemma}
\newtheorem{cor}[thm]{Corollary}
\newtheorem{prop}[thm]{Proposition}
\newtheorem{rem}[thm]{Remark}
\newtheorem{ex}[thm]{Example}
\newtheorem{defi}[thm]{Definition}
\newtheorem{conj}[thm]{Conjecture}
\newcommand{\abs}[1]{\left\vert#1\right\vert}
\newcommand{\rk}{\operatorname{rank}}

\newcommand{\re}{\operatorname{Re}} 
\newcommand{\im}{\operatorname{Im}}

\newcommand{\eq} [1] {\begin{equation}\label{#1}\quad}
\newcommand{\en} {\end{equation}}
\newcommand{\scal}[1]{\langle#1\rangle}
\newcommand{\norm}[1]{\left\Vert#1\right\Vert}
\newcommand{\C}{\mathbb C}\newcommand{\R}{\mathbb R}

\newcommand{\diag}{\operatorname{diag}}

\newcommand{\Id}{I}
\newtheorem{quest}[thm]{Question}

\newcommand{\tr}{\operatorname{tr}}

\newcommand{\Ree}{\operatorname{Re}}

\newcommand{\inn}[2]{\langle #1,#2\rangle}

\theoremstyle{definition}

\theoremstyle{definition}

\makeatletter
\renewcommand*\env@matrix[1][*\c@MaxMatrixCols c]{%
  \hskip -\arraycolsep
  \let\@ifnextchar\new@ifnextchar
  \array{#1}}
\makeatother
\begin{document}
\begin{frontmatter}
\title{On circular Kippenhahn curves and the Gau--Wang--Wu conjecture about nilpotent partial isometries}

\author[eric]{Eric Shen}
\ead{erick.2013@yandex.ru}
\address[eric]{Moscow State University, Moscow, 119991, Russia }

\begin{abstract}
We study linear operators on a finite-dimensional space whose Kippenhahn curves consist of concentric circles centered at the origin. We say that such operators have Circularity property. One class of examples is rotationally invariant operators. To every operator with norm at most one, we associate an infinite sequence of partial isometries and study when Circularity property can be passed back and forth along that sequence. In particular, we introduce a class of operators for which every partial isometry in the aforementioned sequence has Circularity property, and show that this class is broader than the class of rotationally invariant operators. 
As a consequence, every such an operator provides a counterexample to the Gau--Wang--Wu conjecture about nilpotent partial isometries. We also discuss possible refinements of the conjecture. Finally, we propose a way to check whether a matrix is rotationally invariant, suitable for numerical experiments.

\end{abstract}

\end{frontmatter}

\section{Introduction}

Let $A$ be a bounded linear operator acting on a Hilbert space $\mathcal H$. Its \emph{numerical range} is
$$
W(A)=\{\scal{Ax,x}:x\in \mathcal H,\ \norm{x}=1\}.
$$
The numerical range is always convex. In finite dimensions it is the convex hull of a real algebraic curve $C(A)$, the \emph{Kippenhahn curve}, obtained as the envelope of the family
$$
\{e^{-i\theta}(\lambda_j(\theta) + i\R):\theta\in(-\pi,\pi],\ j=1,\dots,n\},$$
where $\lambda_j(\theta)$ are the eigenvalues of the Hermitian pencil
$$
H_A(\theta):=\Ree(e^{-i\theta}A)=\frac12\bigl(e^{-i\theta}A+e^{i\theta}A^*\bigr).
$$
We say that $A$ has \emph{Circularity property} if the spectrum of $H_A(\theta)$ is independent of $\theta$. Equivalently, the Kippenhahn curve of $A$ is a union of circles centered at the origin. 
There are many well known equivalent characterizations, e.g. $A$ has Circularity property if and only if its \textit{numerical shadow} is rotation invariant, see \cite{GaSe12} and \cite{DUNKL20112042}. We briefly recall more direct characterizations in Section 3, see Proposition~\ref{prop:PW} and Proposition~\ref{prop:circ_prop_det}.

There is a stronger symmetry notion.

\begin{defi}
A matrix $A\in M_n$ is \emph{rotationally invariant} if for every $\theta\in\R$ the matrices $A$ and $e^{i\theta}A$ are unitarily similar.
\end{defi}

Rotational invariance of an operator is a strict structural condition, namely, a rotationally invariant operator is unitarily similar to a block shift. See Theorem~\ref{thm:rot_inv_struct}. The paper \cite{GWW} raised three conjectures on partial isometries, including the following.

\begin{conj}[Gau--Wang--Wu]\label{conj:GWW}
If $A$ is a unitarily irreducible nilpotent partial isometry whose numerical range is a circular disc centered at the origin, then $A$ is rotationally invariant.
\end{conj}
One of the other two conjectures was fully proved by Wegert and Spitkovsky in \cite{WeSp}, and there has also been partial progress on the last one, see e.g. \cite{HeSpitSu}, \cite{SuSWe} and recent preprint \cite{PSS}.

Gau, Wang and Wu proved Conjecture~\ref{conj:GWW} in dimensions up to $6$ and established a structural theorem that, in the canonical partial-isometry representation
$$
A=\begin{bmatrix}0&B\\0&C\end{bmatrix},\qquad B^*B+C^*C=\Id,
$$
rotational invariance of $A$ is equivalent to rotational invariance of $C$ \cite[Theorem~4.7]{GWW}. In contrast with the approach of Gau, Wang and Wu, in the present work we adopt the point of view where the primary object of investigation is the contraction $C$, while the whole partial isometry $A$ is considered to be a result of a certain construction $\mathcal{A}(C)$. 

First, in Section~3 we recall how to construct a partial isometry $\mathcal{A}(C)$ from any contraction $C$, see Definition~\ref{defi:const_alpha}. Since a partial isometry is itself a contraction, one may iterate the procedure. That leads to the \textit{partial isometry tower}
$$
T_0:=C,\qquad T_{j+1}:=\mathcal A(T_j),
$$
 We prove that each $T_j$ has Circularity property if and only if the \emph{defect pencil}
$$
H_C(\theta)+\tau(\Id-C^*C)
$$
is isospectral for every real $\tau$, see Theorem~\ref{thm:tower-equiv}. It provides a broad class of counterexamples to Conjecture~\ref{conj:GWW} that have Circularity property, that is, their Kippenhahn curves consist only of concentric circles. See e.g. Example~\ref{ex:tower-counterexample}.  
   
    In Section~4, we investigate further how Circularity property of $C$ relates to that of $\mathcal{A}(C)$. In particular, we show that $\mathcal{A}(C)$ has Circularity property only if $C$ does, provided $C$ has size at most $4$. See Theorem~\ref{thm:Q-implies-P-low-dim} and Corollary~\ref{cor:A-to-C-low-dim} after it. We show that this is not necessarily the case for $C$ of size $5$, see Proposition~\ref{ex:rotationa_A_nonrot_C}.  Moreover, we show that the converse already fails for $C$ of size $4$, see Example~\ref{ex:P-not-Q}. Using results of this section, we construct a unitarily irreducible nilpotent partial isometry that has circular numerical range, yet does not have Circularity property. That is, its Kippenhahn curve is not a union of circles alone. See Proposition~\ref{prop:clean-counterexample}.

In Section 5 we address the following question: Given an operator with the Circularity property, how can one check whether it is rotationally invariant? Existing general criteria can be efficiently used in numerical experiments only for relatively small matrices. We recall existing structural results and suggest a test for rotational invariance suitable for numerical experiments with large operators, see Theorem~\ref{thm:gauwu_poly} and Corollary~\ref{cor:magnitude} after it. As an illustration, we apply our method to construct a family of matrices of arbitrary even size having Circularity property which are not rotationally invariant, see Theorem~\ref{thm:circular-family}.

\section{Preliminaries}\label{sec:prel}

\begin{defi}
    A linear operator $A \in M_n$ is called a partial isometry if it satisfies
\eq{pi} \norm{Ax}=\norm{x} \text{ for all } x\perp\ker A. \en 

\end{defi} 
We begin with a standard reduction.

\begin{prop}[Canonical form of a partial isometry]\label{prop:canonical-form}
Let $A$ be a partial isometry with nontrivial kernel. Then, after a unitary similarity,
$$
A=\begin{bmatrix}0&B\\0&C\end{bmatrix}
\qquad\text{on }\C^m\oplus\C^d,
$$
where $m=\dim\ker A$, $d=\dim(\ker A)^\perp$, and
\begin{equation}\label{eq:defect-relation}
B^*B+C^*C=\Id_d.
\end{equation}
Conversely, every matrix of this form satisfying \eqref{eq:defect-relation} is a partial isometry.
\end{prop}

The bulk of our work is closely related to the following result of Gau, Wang, and Wu.

\begin{thm}[\cite{GWW}, Theorem~4.7]\label{thm:GWW-rot}
Let
$$
A=\begin{bmatrix}0&B\\0&C\end{bmatrix}
$$
be a partial isometry in canonical form. Then $A$ is rotationally invariant if and only if $C$ is rotationally invariant.
\end{thm}

For later use we record the standard irreducibility criterion.

\begin{prop}[\cite{GWW}, Proposition~2.6]\label{prop:irr}
The matrix
$$
A=\begin{bmatrix}0&B\\0&C\end{bmatrix}
$$
is unitarily irreducible if and only if $C$ is unitarily irreducible and $B$ has full row rank.
\end{prop}

We note that $||C|| \leq 1$. For brevity, throughout the paper an operator satisfying that inequality is called \textit{a contraction}.
\section{Partial isometry towers and defect-pencil Circularity}\label{sec:towers}
Is there a genuinely weaker condition than rotational invariance that can be passed back and forth from a contraction $C$ to its associated partial isometry $\mathcal A(C)$? In this section, we answer this question positively. That in turn yields a class of counterexamples to the original Gau-Wang-Wu conjecture on nilpotent partial isometries.

The main idea is to iterate the passage from a contraction $C$ to its associated partial isometry $\mathcal A(C)$. This produces a new circularity notion which sits naturally between the ordinary Circularity property of $C$, equivalently isospectrality of $H_C(\theta)$, and the stronger rotational invariance of $C$ itself.

\begin{defi} \label{defi:const_alpha}
Let $C\in M_n(\C)$ be a contraction and set
$$
D_C:=\Id-C^*C\ge 0.
$$
Define a map $$B_C:\C^n\to \im D_C,$$ such that $B_C$ acts on $\im D_C \subseteq \C^n$ as the square root $D_C^{1/2}$ and sends $(\im D_C)^\perp$ to zero. The associated partial isometry of $C$ is
$$
\mathcal A(C):=\begin{bmatrix}0&B_C\\0&C\end{bmatrix}\in M_{n+\rk B_C}(\C).
$$
The partial isometry tower of $C$ is the sequence
$$
T_0:=C,
\qquad
T_{j+1}:=\mathcal A(T_j),\quad j\ge 0.
$$
\end{defi}
\begin{rem}
Here we need to mention that this definition amounts to a choice of coordinates. Namely, one may replace $B_C$ with some other matrix $B_C'$ satisfying $B_C'^*B_C' + C^*C = I$, obtaining another partial isometry. But provided $B_C'$ has the same row rank as $B_C$, this new partial isometry will be unitarily similar to $\mathcal{A}(C)$. 
\end{rem}

Moving from $C$ to $\mathcal{A}(C)$ does not add any additional information on $C$ from the point of view of unitary equivalence, as shown by the following:
\begin{prop}
    $\mathcal{A}(C)$ is unitarily similar to $\mathcal{A}(C')$ if and only if $C$ and $C'$ are unitarily similar.
\end{prop}
\begin{proof}
    
Assume $C'=UCU^*$ for some fixed unitary $U$. Then one can find a block unitary 
$$W:=U_{1}\oplus U$$
such that
\[
W\mathcal A(C)W^*
=
\begin{bmatrix}
0&U_{1}B_CU^*\\
0&UCU^*
\end{bmatrix}
=
\begin{bmatrix}
0&B_{C'}\\
0&C'
\end{bmatrix}
=\mathcal A(C').
\]

Now note that $C$ is a matrix of an operator $\mathcal{C} = \Pi \mathcal{A}(C)|_{(\ker \mathcal{A}(C))^\perp} $, where $\Pi$ is the orthogonal projector onto $(\ker \mathcal{A}(C))^\perp$. Therefore, whenever $\mathcal{A}(C)$ is unitarily similar to $\mathcal{A}(C')$, the same holds for $C$ and $C'$.

\end{proof}
The following Proposition shows what happens if one iterates the construction:

\begin{prop}\label{prop:tower-shape}
Let $C, B_C$, $D_C$, $T_j$ for $j \geq 0$ be as in Definition~\ref{defi:const_alpha}, and denote $d = \rk D_C$. Then in suitable orthonormal coordinates,
    $$
    T_j=\begin{bmatrix}
    0&\Id_d&0&\cdots&0&0\\
    0&0&\Id_d&\cdots&0&0\\
    \vdots& \vdots&\ddots&\ddots&\vdots&\vdots\\
    0&0&\cdots&0&\Id_d&0\\
    0&0&\cdots&0&0&B_C\\
    0&0&\cdots&0&0&C
    \end{bmatrix},
    $$
    where the first $j-1$ block rows and columns are $d\times d$.

\end{prop}

\begin{proof}
By definition,
$$
T_1=\begin{bmatrix}0&B_C\\0&C\end{bmatrix},
\qquad
T_1^*T_1=\begin{bmatrix}0&0\\0&B_C^*B_C+C^*C\end{bmatrix}=\diag(0_d,I_n).
$$
Thus, $B_{T_1} = \begin{bmatrix}
    I_d & 0
\end{bmatrix}$, and one obtains
$$T_2 = \mathcal{A}(T_1)=\begin{bmatrix}
    0 & I_d & 0 \\
    0 & 0  & B_C \\
    0 & 0 & C
\end{bmatrix}.$$

Now assume that for some $j\ge 2$ we have, in suitable coordinates,
$$
T_j=\begin{bmatrix}
0&\Id_d&0&\cdots&0&0\\
0&0&\Id_d&\cdots&0&0\\
\vdots&&\ddots&\ddots&\vdots&\vdots\\
0&0&\cdots&0&\Id_d&0\\
0&0&\cdots&0&0&B_C\\
0&0&\cdots&0&0&C
\end{bmatrix},
$$
where the first $j-1$ block rows and columns are $d\times d$.

Therefore
$$
T_j^*T_j=\diag(0_d,\Id_{(j-1)d+n}),
\qquad
\Id-T_j^*T_j=\diag(\Id_d,0_{(j-1)d+n}),
$$
Hence once again
$$
B_{T_j}=\begin{bmatrix}\Id_d&0&\cdots&0\end{bmatrix},
$$
and we see that
$$
T_{j+1}=\mathcal A(T_j)=\begin{bmatrix}0&B_{T_j}\\0&T_j\end{bmatrix}
$$
has exactly the same block pattern with one additional $I_d$ block. 
\end{proof}
\begin{rem}\label{rem:Schaffer}
Proposition~\ref{prop:tower-shape} is closely related to the classical Schäffer
isometric dilation of a contraction, see \cite{Schaffer}. Let $E=\overline{\im \Delta_C}$, where
$\Delta_C=(I-C^*C)^{1/2}$, and let
\[
V_C(h,\xi_1,\xi_2,\ldots)=(Ch,\Delta_Ch,\xi_1,\xi_2,\ldots),
\qquad
h\in\C^n,\ \xi_j\in E,
\]
be the Schäffer isometry on $\C^n\oplus \ell^2(\mathbb N,E)$.
If one reverses the order of the block summands in the matrix from
Proposition~\ref{prop:tower-shape}, then $T_j$ becomes the compression of $V_C$
to the finite-dimensional subspace
\[
\C^n\oplus E\oplus\cdots\oplus E
\qquad (j \text{ copies of } E).
\]
Thus the partial-isometry tower may be viewed as a sequence of finite-dimensional truncations
of the Schäffer dilation. In other words, after this reordering, the adjoints of
these truncations form an inductive system whose direct limit is $V_C^*$.
\end{rem}

Moreover, the Kippenhahn polynomial of each $T_j$ is determined by the base matrix, $T_0$, as shown by the following Proposition:

\begin{prop}\label{prop:tower-det}
As before, let $D=\Id-C^*C$ and let $d=\rk D$. Define polynomials $\phi_{-1}=0$, $\phi_0=1$ and
\eq{eq:phi_rec}
\phi_{m+1}(z,r)=z\phi_m(z,r)-r^2\phi_{m-1}(z,r).
\en
Then for every $j\ge 1$,
\eq{eq:tower-det}
\det\bigl(z\Id-2rH_{T_j}(\theta)\bigr)
=
\phi_j(z,r)^d\det\!\left(z\Id-2rH_C(\theta)-r^2\frac{\phi_{j-1}(z,r)}{\phi_j(z,r)}D\right).
\en
\end{prop}

\begin{proof}
Fix $j\ge1$. By Proposition~\ref{prop:tower-shape}, in suitable coordinates $T_j$ has the stated block form. Set
\[
K_j(\theta):=z\Id-2rH_{T_j}(\theta)=z\Id-r e^{-i\theta}T_j-r e^{i\theta}T_j^*.
\]
 Let
\[
U_j(\theta):=\diag\bigl(\Id_d,e^{i\theta}\Id_d,e^{2i\theta}\Id_d,\dots,e^{ij\theta}\Id_n\bigr).
\]
 Since the $(k,k+1)$ block of $T_j$ is $\Id_d$ for $1\le k\le j-1$, the $(j,j+1)$ block is $B_C$, and the $(j+1,j+1)$ block is $C$, a direct block computation gives
\[
U_j(\theta)^*\bigl(r e^{-i\theta}T_j\bigr)U_j(\theta)=
\begin{bmatrix}
0&r\Id_d&0&\cdots&0&0\\
0&0&r\Id_d&\cdots&0&0\\
\vdots&&\ddots&\ddots&\vdots&\vdots\\
0&0&\cdots&0&r\Id_d&0\\
0&0&\cdots&0&0&rB_C\\
0&0&\cdots&0&0&r e^{-i\theta}C
\end{bmatrix},
\]
because for $1\le k\le j-1$,
\[
e^{-i(k-1)\theta}\,r e^{-i\theta}\,e^{ik\theta}=r,
\]
and similarly for the $B_C$-block,
\[
e^{-i(j-1)\theta}\,r e^{-i\theta}\,e^{ij\theta}=r.
\]
Taking adjoints gives
\[
U_j(\theta)^*\bigl(r e^{i\theta}T_j^*)U_j(\theta)=
\begin{bmatrix}
0&0&0&\cdots&0&0\\
r\Id_d&0&0&\cdots&0&0\\
0&r\Id_d&0&\ddots&\vdots&\vdots\\
\vdots&&\ddots&\ddots&0&0\\
0&0&\cdots&r\Id_d&0&0\\
0&0&\cdots&0&rB_C^*&r e^{i\theta}C^*
\end{bmatrix}.
\]
Therefore
\[
M_j(\theta):=U_j(\theta)^*K_j(\theta)U_j(\theta)
\]
is exactly
\[
M_j(\theta)=\begin{bmatrix}A_j&E\\E^*&S(\theta)\end{bmatrix},
\]
where
\[
A_j=J_j\otimes \Id_d,
\qquad
S(\theta)=z\Id_n-r e^{-i\theta}C-r e^{i\theta}C^*=z\Id_n-2rH_C(\theta),
\]
with
\[
J_j=\begin{bmatrix}
z&-r&0&\cdots&0\\
-r&z&-r&\cdots&0\\
0&-r&z&\ddots&\vdots\\
\vdots&&\ddots&\ddots&-r\\
0&\cdots&0&-r&z
\end{bmatrix},
\]
and
\[
E=
\begin{bmatrix}
0\\
\vdots\\
0\\
-rB_C
\end{bmatrix}
=(e_j\otimes(-r\Id_d))B_C,
\]
a $jd\times n$ block column whose only nonzero $d\times n$ block is the last one.

The determinant of $J_j$ satisfies the recurrence \eqref{eq:phi_rec}, so
\[
\det J_j=\phi_j(z,r),
\qquad
\det A_j=(\det J_j)^d=\phi_j(z,r)^d.
\]
Assume first that $\phi_j(z,r)\neq0$, equivalently $A_j$ is invertible. Then the Schur complement gives
\[
\det M_j(\theta)=\det A_j\,\det\bigl(S(\theta)-E^*A_j^{-1}E\bigr).
\]
Now
\[
A_j^{-1}=J_j^{-1}\otimes \Id_d,
\]
so only the $(j,j)$ entry of $J_j^{-1}$ contributes when we sandwich $A_j^{-1}$ between $E^*$ and $E$. Indeed, writing $e_j=(0,\dots,0,1)^T\in\C^j$,
\[
E=(e_j\otimes(-r\Id_d))B_C.
\]
Hence
\begin{align*}
E^*A_j^{-1}E
&=B_C^*(e_j^*\otimes(-r\Id_d))(J_j^{-1}\otimes \Id_d)(e_j\otimes(-r\Id_d))B_C
\\
&\qquad\qquad\ =r^2(e_j^*J_j^{-1}e_j)B_C^*B_C\
=r^2(J_j^{-1})_{jj}B_C^*B_C.
\end{align*}
By the cofactor formula,
\[
(J_j^{-1})_{jj}=\frac{\det J_{j-1}}{\det J_j}=\frac{\phi_{j-1}(z,r)}{\phi_j(z,r)}.
\]
Since $B_C^*B_C=D$, we obtain
\[
E^*A_j^{-1}E=r^2\frac{\phi_{j-1}(z,r)}{\phi_j(z,r)}D.
\]
Substituting into the Schur complement identity yields \eqref{eq:tower-det} whenever $\phi_j(z,r)\neq0$. Finally, both sides of \eqref{eq:tower-det} are polynomials in $z$, so the identity holds for all $z$.
\end{proof}

Writing $z=2ru$ and using the Chebyshev polynomials of the second kind, one gets
\begin{equation}\label{eq:cheb-form}
\det\bigl(z\Id-2rH_{T_j}(\theta)\bigr)
=(2r)^n r^{jd}U_j(u)^d\det\bigl(u\Id-H_C(\theta)-\tau_j(u)D\bigr),
\end{equation}
where
$$
\tau_j(u):=\frac{U_{j-1}(u)}{2U_j(u)}.
$$
This motivates the central notion.

\begin{defi}\label{def:defect-pencil}
A contraction $C$ has \emph{defect-pencil Circularity} if for every real $\tau$ the spectrum of
$$
H_C(\theta)+\tau(\Id-C^*C)
$$
is independent of $\theta$.
\end{defi}

\begin{thm}\label{thm:tower-equiv}
For any contraction $C$, the following are equivalent:
\begin{enumerate}[label=\textup{(\roman*)}]
    \item every level of the partial isometry tower of $C$ has Circularity property;
    \item $C$ has defect-pencil Circularity.
\end{enumerate}
\end{thm}

\begin{proof}
Assume \textup{(ii)}. Then the determinant on the right-hand side of \eqref{eq:cheb-form} is independent of $\theta$ for every $j$ and $u$, so each
$$
\det\bigl(z\Id-2rH_{T_j}(\theta)\bigr)
$$
is independent of $\theta$. Hence every $T_j$ has Circularity property.

Assume conversely that every $T_j$ has Circularity property. Fix $u\in\R$ with $|u|>1$. Then $U_j(u)\neq0$ for all $j$, and by \eqref{eq:cheb-form} the polynomial
$$
F_{\theta}^{(u)}(\tau):=\det\bigl(u\Id-H_C(\theta)-\tau D\bigr)
$$
does not depend on $\theta$ at all points $\tau=\tau_j(u)$. Since $\deg_\tau F_{\theta}^{(u)}\le d$, to conclude that the same holds for every $\eta$ it suffices to know that $\tau_j(u)$ are pairwise distinct.

If $u=\cosh t>1$, then
$$
U_j(u)=\frac{\sinh((j+1)t)}{\sinh t},
\qquad
\tau_j(u)=\frac12\frac{\sinh(jt)}{\sinh((j+1)t)}.
$$
If $u=-\cosh t<-1$, then
$$
U_j(u)=(-1)^j\frac{\sinh((j+1)t)}{\sinh t},
\qquad
\tau_j(u)=-\frac12\frac{\sinh(jt)}{\sinh((j+1)t)}.
$$
In either case the numbers $\tau_j(u)$ are pairwise distinct, because
$$
\frac{\sinh((j+1)t)}{\sinh((j+2)t)}-\frac{\sinh(jt)}{\sinh((j+1)t)}
=\frac{\sinh^2 t}{\sinh((j+1)t)\sinh((j+2)t)}>0.
$$
Therefore $F_{\theta}^{(u)}$ agrees with $F_0^{(u)}$ at infinitely many distinct real points and hence is identical to it. This proves defect-pencil Circularity for every $|u|>1$.

Now fix $\tau\in\R$. For each $\theta$, the function
$$
\nu\mapsto \det\bigl(\nu\Id-H_C(\theta)-\tau D\bigr)
$$
is a polynomial in $\nu$ of degree $n$. It agrees with the corresponding polynomial at $\theta=0$ for all $|\nu|>1$, hence for all $\nu\in\R$. Thus defect-pencil Circularity holds for every real $\tau$.
\end{proof}

Since the defect rank equals the degree of the determinant in $\tau$, only finitely many levels matter.

\begin{cor}\label{cor:finitely-many}
Let $d=\rk(\Id-C^*C)$. Then the following are equivalent:
\begin{enumerate}[label=\textup{(\roman*)}]
    \item every level of the partial isometry tower of $C$ is circular;
    \item $T_1,T_2,\dots,T_{d+1}$ have Circularity property.
\end{enumerate}
If one prefers to include the base matrix, then $T_0,T_1,\dots,T_d$ also suffice.
\end{cor}

\begin{proof}
Only the implication (ii) $\Longrightarrow$ (i) needs proof. Fix $u$ with $|u|>1$. By Theorem~\ref{thm:tower-equiv}, it is enough to show that
$$
F_\theta^{(u)}(\tau)=\det\bigl(u\Id-H_C(\theta)-\tau D\bigr)
$$
is independent of $\theta$ for all $\tau$. The polynomial $F_\theta^{(u)}$ has degree at most $d$ in $\tau$, and the levels $T_1,\dots,T_{d+1}$ provide the $d+1$ distinct values $\tau_1(u),\dots,\tau_{d+1}(u)$, so interpolation determines the entire polynomial.
\end{proof}

There is also a trace-word reformulation.

\begin{defi}
Fix a contraction $C$ and set $D=\Id-C^*C$. For integers $k\ge 1$, $0\le m\le k$ and $0\le \ell\le k-m$, let $\mathcal W_{k,m,\ell}(C,C^*,D)$ be the set of all words of length $k$ in the alphabet $\{C,C^*,D\}$ having exactly $m$ letters $D$ and exactly $\ell$ letters $C^*$. Define
$$
\Omega_{k,m,\ell}(C):=\sum_{w\in \mathcal W_{k,m,\ell}(C,C^*,D)}\tr w(C, C^*, D).
$$
\end{defi}

\begin{thm}\label{thm:trace-word}
For a contraction $C\in M_n$, the following are equivalent:
\begin{enumerate}[label=\textup{(\roman*)}]
    \item every level of the partial isometry tower of $C$ has Circularity property;
    \item for every $1\le k\le n$, every $0\le m\le k$, and every $\ell$ with $2\ell\neq k-m$,
    $$
    \Omega_{k,m,\ell}(C)=0.
    $$
\end{enumerate}
\end{thm}

\begin{proof}
By Theorem~\ref{thm:tower-equiv}, it suffices to characterize defect-pencil Circularity.

Fix $\tau\in\R$ and write
$$
M_\tau(\theta):=H_C(\theta)+\tau D
=\frac12\bigl(e^{-i\theta}C+e^{i\theta}C^*\bigr)+\tau D.
$$
The spectrum of $M_\tau(\theta)$ is independent of $\theta$ if and only if
$$
\tr(M_\tau(\theta)^k),\qquad 1\le k\le n,
$$
are independent of $\theta$, by Newton identities.

Now expand $M_\tau(\theta)^k$. A word with $m$ letters $D$, $\ell$ letters $C^*$, and $k-m-\ell$ letters $C$ contributes the scalar factor
$$
2^{-(k-m)}\tau^m e^{i(2\ell-(k-m))\theta}.
$$
Therefore the coefficient of $\tau^m e^{i(2\ell-(k-m))\theta}$ in $\tr(M_\tau(\theta)^k)$ equals $2^{-(k-m)}\Omega_{k,m,\ell}(C)$. Independence of $\theta$ is therefore equivalent to the vanishing of every nonzero Fourier coefficient, i.e.
$$
\Omega_{k,m,\ell}(C)=0\qquad\text{whenever }2\ell\neq k-m.
$$
This proves the equivalence.
\end{proof}

We compare defect-pencil Circularity with the following characterization (see e.g. \cite{PoonWoerd}).

\begin{prop}[{\cite[Theorems~1.1 and~3.1]{PoonWoerd}}]\label{prop:PW}
Let $B\in M_n$.
\begin{enumerate}[label=\textup{(\roman*)}]
    \item $B$ has Circularity property if and only if for every $1\le k\le n$ and every $0\le \ell<k/2$,
    $$
    \sum_{\substack{|w|=k\\ \#(B^*\text{ in }w)=\ell}} \tr w(B,B^*)=0.
    $$
    Equivalently, $H_B(\theta)$ is unitarily similar to $H_B(0)$ for all $\theta$.
    \item $B$ is rotationally invariant if and only if
    $$
    \tr w(B,B^*)=0
    $$
    for every word $w$ in $B,B^*$ with different numbers of $B$ and $B^*$.
\end{enumerate}
\end{prop}

Thus defect-pencil Circularity sits naturally between the two conditions:
$$
\text{rotational invariance}
\Longrightarrow
\text{defect-pencil Circularity}
\Longrightarrow
\text{Circularity property}.
$$
Indeed, if $U_\theta^*CU_\theta=e^{i\theta}C$, then also $U_\theta^*D\,U_\theta=D$, so
$$
U_\theta^*\bigl(H_C(0)+\tau D\bigr)U_\theta=H_C(\theta)+\tau D.
$$
On the level of trace identities, rotational invariance forces vanishing of every individual unbalanced trace, defect-pencil Circularity forces vanishing of the aggregated defect-word sums $\Omega_{k,m,\ell}$, and ordinary Circularity forces vanishing of the aggregated unbalanced sums in the alphabet $\{C,C^*\}$.

The same comparison has a determinantal side, see e.g. Theorem 1.1 in $\cite{PoonWoerd}$.

\begin{prop}\label{prop:circ_prop_det}
An operator $C$ has Circularity property if and only if
   $$\det\bigl(u\Id-H_C(\theta))$$
    is independent of $\theta$.
\end{prop}

At the same time we have shown that defect-pencil Circularity is equivalent to $\theta$-independence of
$$
\det\bigl(u\Id-H_C(\theta)-\tau(\Id-C^*C)\bigr)
$$
for all real $u,\tau$. To our knowledge, there is no similar characterization for rotational invariance. This leads to the following natural question.

\begin{quest}\label{quest:det-rot}
Does rotational invariance admit a comparably natural determinantal reformulation, analogous to the determinant criteria for Circularity property and for defect-pencil Circularity?
\end{quest}

To conclude this section, we give an example of a contraction that has defect-pencil Circularity. 

\begin{ex}\label{ex:tower-counterexample}
The nilpotent contraction
\begin{equation}\label{eq:tower-C}
C=
\begin{pmatrix}
0 & 1/2 & -1/5 & 0\\
0 & 0 & 1/2 & 1/5\\
0 & 0 & 0 & 1/2\\
0 & 0 & 0 & 0
\end{pmatrix}
\end{equation}
is not rotationally invariant, but every operator in its partial isometry tower has Circularity property.
\end{ex}

\begin{proof}
Indeed, the unbalanced trace
$$
\tr(C^2C^*C^2C^*)=\frac1{200}\neq 0
$$
shows that $C$ is not rotationally invariant by Proposition~\ref{prop:PW}(ii). On the other hand, for
$$
D=\Id-C^*C=
\begin{pmatrix}
1 & 0 & 0 & 0\\
0 & 3/4 & 1/10 & 0\\
0 & 1/10 & 71/100 & -1/10\\
0 & 0 & -1/10 & 71/100
\end{pmatrix},
$$
a direct symbolic computation shows that
$$
\Delta(\theta;\alpha,s,t)=\det\bigl(\alpha\Id+s(e^{-i\theta}C+e^{i\theta}C^*)+tD\bigr)
$$

is independent of $\theta$. More precisely,
\begin{align*}
\Delta(\theta;\alpha,s,t)
&=\frac1{40000}\Bigl(
40000\alpha^4+126800\alpha^3 t-33200\alpha^2s^2+148764\alpha^2t^2\\
&\qquad\qquad\ -54672\alpha s^2t+76503\alpha t^3+3364 s^4-22097 s^2t^2+14539 t^4
\Bigr).
\end{align*}
for all $\alpha,s,t\in\C$. Equivalently,
$$
\det\bigl(u\Id-H_C(\theta)-\tau D\bigr)
$$
is independent of $\theta$ for all real $u,\tau$. Thus $C$ has defect-pencil Circularity. The conclusion now follows from Theorem~\ref{thm:tower-equiv}.
\end{proof}

In particular, if $A=\mathcal A(C)$ is the first associated partial isometry of \eqref{eq:tower-C}, then both $A$ and $C$ have Circularity property, yet neither is rotationally invariant. In other words, $\mathcal{A}(C)$ is a counterexample to the Gau-Wang-Wu conjecture about nilpotent partial isometries.

\section{A relaxation of the Gau-Wang-Wu conjecture}\label{sec:PQ}
As before, let $C$ be a contraction and $A = \mathcal{A}(C)$ its associated partial isometry. 
The (counter)-example we just described suggests that one may try to relax Conjecture~\ref{conj:GWW} replacing \textit{rotationally invariant} with \textit{has Circularity property}. Among other things, in this section we show that in general this relaxed conjecture does not hold either, see Proposition~\ref{prop:clean-counterexample}. We also show that neither Circularity of \(C\) nor Circularity of \(A\) guarantees the same property for the other matrix. That is, in this section we investigate how much Circularity of $A$ forces Circularity of $C$, and vice versa.

 Define
$$
P_C(x,y):=\det(\Id-xC-yC^*),\qquad
Q_C(x,y):=\det(\Id-xC-yC^*-xyD),
$$
where as before $D =  I - C^*C$.
\begin{prop}\label{prop:bridge}
For every $z\in\C\setminus\{0\}$ and every $\theta\in\R$,
\begin{equation}\label{eq:bridge-Q}
\det\bigl(z\Id- e^{-i\theta}A-e^{i\theta}A^*\bigr)
=z^{m+d}\,Q_C\!\left(\frac{e^{-i\theta}}{z},\frac{e^{i\theta}}{z}\right).
\end{equation}
Likewise,
\begin{equation}\label{eq:bridge-P}
\det\bigl(z\Id- e^{-i\theta}C-e^{i\theta}C^*\bigr)
=z^d\,P_C\!\left(\frac{e^{-i\theta}}{z},\frac{e^{i\theta}}{z}\right).
\end{equation}
Consequently:
\begin{enumerate}[label=\textup{(\roman*)}]
    \item $A$ has Circularity property if and only if $Q_C(x,y)$ depends only on $xy$;
    \item $C$ has Circularity property if and only if $P_C(x,y)$ depends only on $xy$.
\end{enumerate}
\end{prop}

\begin{proof}
Starting from
$$
z\Id- e^{-i\theta}A-e^{i\theta}A^*
=
\begin{bmatrix}
z\Id_m & -e^{-i\theta}B\\
-e^{i\theta}B^* & z\Id_d-e^{-i\theta}C-e^{i\theta}C^*
\end{bmatrix},
$$
we apply the Schur complement with respect to the upper-left block $z\Id_m$ and use $B^*B=D$:
\begin{align*}
\det\bigl(z\Id- e^{-i\theta}A-e^{i\theta}A^*\bigr)
&=z^m\det\left(z\Id_d-e^{-i\theta}C-e^{i\theta}C^*-z^{-1}D\right)\\
&=z^{m+d}\det\left(\Id_d-\frac{e^{-i\theta}}{z}C-\frac{e^{i\theta}}{z}C^*-\frac1{z^2}D\right).
\end{align*}
Since
$$
\frac1{z^2}=\frac{e^{-i\theta}}{z}\cdot \frac{e^{i\theta}}{z},
$$
this is exactly \eqref{eq:bridge-Q}. Formula \eqref{eq:bridge-P} is immediate. Finally, dependence on $\theta$ enters only through the phases of $x=e^{-i\theta}/z$ and $y=e^{i\theta}/z$ with fixed product $xy=z^{-2}$, so isospectrality for all $\theta$ is equivalent to radiality in the variable pair $(x,y)$.
\end{proof}

 For $t\in\R$ define
$$
H(t):=\frac12\bigl(e^{-i t}C+e^{i t}C^*\bigr),
\qquad
K(t):=\frac{i}{2}\bigl(e^{i t}C^*-e^{-i t}C\bigr).
$$
Consider two trace functions
\begin{align}
T_1(t,\mu)&:=\tr\!\left((\mu\Id-H(t))^{-1}K(t)\right),\label{eq:T1}\\
T_2(t,\lambda)&:=\tr\!\left((\lambda^2\Id-\lambda H(t)-\tfrac14(\Id-C^*C))^{-1}\lambda K(t)\right).\label{eq:T2}
\end{align}

\begin{lem}\label{lem:T1T2-radiality}
\begin{enumerate}
\item \(T_1(t,\mu)\equiv 0\) if and only if \(P_C(x,y)\) depends only on \(xy\).
\item \(T_2(t,\lambda)\equiv 0\) if and only if \(Q_C(x,y)\) depends only on \(xy\).
\end{enumerate}
\end{lem}

\begin{proof}
We prove the statement for $P_C$; the proof for $Q_C$ is analogous.

For a polynomial $R(x,y)\in\C[x,y]$, the condition $R\in\C[xy]$ is equivalent to saying
that for every fixed $\rho\in\C$, the function
\[
t\longmapsto R(\rho e^{-it},\rho e^{it})
\]
is independent of $t$.
Indeed, if $R(x,y)=\widetilde R(xy)$, then this is immediate. Conversely, if
\[
R(x,y)=\sum_{p,q} a_{pq}x^py^q,
\]
then
\[
R(\rho e^{-it},\rho e^{it})
=
\sum_{p,q} a_{pq}\rho^{p+q}e^{i(q-p)t}.
\]
If this is independent of $t$ for every $\rho$, then all Fourier coefficients with $p\neq q$
must vanish, hence $R\in\C[xy]$.

Apply this to $R=P_C$. For $\mu\in\C$, set
\[
x=\frac{e^{-it}}{2\mu},
\qquad
y=\frac{e^{it}}{2\mu}.
\]
Then, by the definition of $H(t)$,
\[
I-xC-yC^*
=
I-\frac1{2\mu}(e^{-it}C+e^{it}C^*)
=
I-\frac1\mu H(t),
\]
and therefore
\[
P_C\!\left(\frac{e^{-it}}{2\mu},\frac{e^{it}}{2\mu}\right)
=
\det\!\left(I-\frac1\mu H(t)\right)
=
\mu^{-n}\det(\mu I-H(t)).
\]
Differentiate with respect to $t$. Jacobi's formula gives
\[
\frac{d}{dt}
P_C\!\left(\frac{e^{-it}}{2\mu},\frac{e^{it}}{2\mu}\right)
=
-
P_C\!\left(\frac{e^{-it}}{2\mu},\frac{e^{it}}{2\mu}\right)\,T_1(t,\mu).
\]

Assume first that $T_1(t,\mu)\equiv 0$. Then the derivative above vanishes, so
$t\mapsto P_C(e^{-it}/(2\mu),e^{it}/(2\mu))$ is constant for every $\mu$. By the first
paragraph, this means that $P_C\in\C[xy]$.

Conversely, assume that $P_C\in\C[xy]$. Then
\[
t\longmapsto P_C\!\left(\frac{e^{-it}}{2\mu},\frac{e^{it}}{2\mu}\right)
\]
is independent of $t$, so its derivative is zero. Thus
\[
P_C\!\left(\frac{e^{-it}}{2\mu},\frac{e^{it}}{2\mu}\right)\,T_1(t,\mu)=0.
\]
For fixed $t$, this holds for all sufficiently large $\mu$, and for such $\mu$ the factor
$P_C(e^{-it}/(2\mu),e^{it}/(2\mu))$ is nonzero. Hence $T_1(t,\mu)=0$ for all sufficiently
large $\mu$, and therefore, being a rational function of $\mu$, it vanishes identically.
So $T_1(t,\mu)\equiv 0$.

\end{proof}
In small dimensions, $\mathcal{A} (C)$ has Circularity property only if $C$ does, as given by the following theorem.

\begin{thm}\label{thm:Q-implies-P-low-dim}
Let $C$ be a nilpotent contraction of size at most $4$. If $T_2 \equiv 0$, then $T_1 \equiv 0$.
\end{thm}
\begin{proof}

Let us expand $T_2$ at $\lambda=\infty$. That gives
\begin{equation}\label{eq:T2Laurent}
T_2(t,\lambda)=\sum_{N\ge 0}\lambda^{-(N+1)}A_N(t).
\end{equation}
The first coefficients are:
\begin{align}
A_0&=\tr(K),\label{eq:A0}\\
A_1&=\tr(HK),\label{eq:A1}\\
A_2&=\frac14\tr(K)+\frac34\tr(H^2K)-\frac14\tr(K^3),\label{eq:A2}\\
A_3&=\frac12\tr(HK)+\frac12\tr(H^3K)-\frac12\tr(HK^3),\label{eq:A3}
\end{align}
We also need $A_4$ and $A_5$ for what follows, yet they are rather cumbersome:
\begin{align}
A_4={}&\frac1{16}\tr(K)+\frac58\tr(H^2K)-\frac18\tr(K^3)
+\frac5{16}\tr(H^4K)\nonumber\\
&-\frac5{16}\tr(H^2K^3)-\frac5{16}\tr(HKHK^2)
+\frac1{16}\tr(K^5),\label{eq:A4}\\[0.5em]
A_5={}&\frac3{16}\tr(HK)+\frac58\tr(H^3K)-\frac38\tr(HK^3)
+\frac3{16}\tr(H^5K)\nonumber\\
&-\frac3{16}\Big(\tr(H^3K^3)+\tr(H^2KHK^2)+\tr(H^2K^2HK)\Big)
-\frac1{16}\tr(HKHKHK)\nonumber\\
&+\frac3{16}\tr(HK^5).\label{eq:A5}
\end{align}

Similarly, let us expand \(T_1\) at infinity:
$$
(\mu\Id-H(t))^{-1}
=
\frac1\mu\Bigl(\Id-\mu^{-1}H(t)\Bigr)^{-1}
=
\frac1\mu\sum_{n=0}^\infty \frac{H(t)^n}{\mu^n}.
$$
Therefore
$$
T_1(t,\mu)
=
\tr\!\left((\mu\Id-H(t))^{-1}K(t)\right)
=
\sum_{n=0}^\infty \frac{\tr(H(t)^nK(t))}{\mu^{n+1}}.
$$

In particular,
$$
T_1(t,\mu)\equiv 0
\quad\Longleftrightarrow\quad
\tr(H(t)^nK(t))=0
\ \text{for all }n\ge 0.
$$

We now show that if $T_2\equiv 0$, then
\eq{eq:T1_moments_dead}
\tr(K)=\tr(HK)=\tr(H^2K)=\tr(K^3)=\tr(H^3K)=\tr(HK^3)=0,
\en
identically in $t$.
Indeed, from $A_0=0$ and $A_1=0$:
$$
\tr(K)=0,\qquad \tr(HK)=0.
$$

From $A_2=0$:
$$
\tr(K^3)=3\,\tr(H^2K).
$$

From $A_4=0$ and the imaginary part of $\tr((H-iK)^5)=0$, one gets two formulas for
$$
\tr(H^2K^3)+\tr(HKHK^2),
$$
whose difference forces
$$
\tr(H^2K)=0,
$$
hence $\tr(K^3)=0$.

From $A_3=0$:
$$
\tr(H^3K)=\tr(HK^3).
$$

From $A_5=0$ and the imaginary part of $\tr((H-iK)^6)=0$, one gets a second linear relation between $\tr(H^3K)$ and $\tr(HK^3)$, forcing both to vanish.

In dimension $n\le 4$, Cayley--Hamilton expresses all higher powers $H^m$ as linear combinations of $\Id,H,H^2,H^3$. Since equalities \eqref{eq:T1_moments_dead} hold, all moments $\tr(H^mK)$ vanish, hence $T_1\equiv 0$.

\end{proof}

\begin{cor}\label{cor:A-to-C-low-dim}
Let
$$
A=\begin{bmatrix}0&B\\0&C\end{bmatrix}
$$
be a unitarily irreducible nilpotent partial isometry with $C$ of size at most $4$. If $A$ has Circularity property, then $C$ has Circularity property.
\end{cor}

The converse fails already in size $4$.

\begin{ex}\label{ex:P-not-Q}
Consider
$$
C=
\begin{bmatrix}
0&-\frac{1+i}{4}&-\frac{1}{4}&0 \\
0&0&\frac{1}{4}&\frac{-1+i}{4}\\
0&0&0&\frac{1}{4}\\
0&0&0&0\\
\end{bmatrix}.
$$
It is nilpotent and $\norm{C}<1$. A direct computation gives
$$
P_C(x,y)=1-\frac{7}{16}xy+\frac{1}{32}x^2y^2,
$$
so $P_C\in\C[xy]$ and $C$ has Circularity property. However, the Laurent coefficient $A_6$ for $T_2$ is
$$
A_6=-\frac{1}{65536}\neq 0,
$$
so $T_2\not\equiv 0$, equivalently $Q_C\notin\C[xy]$. Therefore the associated partial isometry built from $C$ does \emph{not} have Circularity property. See Figure~\ref{fig:noncircA_circC} for an approximate plot of the Kippenhahn curve of $\mathcal{A}(C)$.
\end{ex}
\begin{figure}[H]
    \centering
    \includegraphics[width=0.9\linewidth]{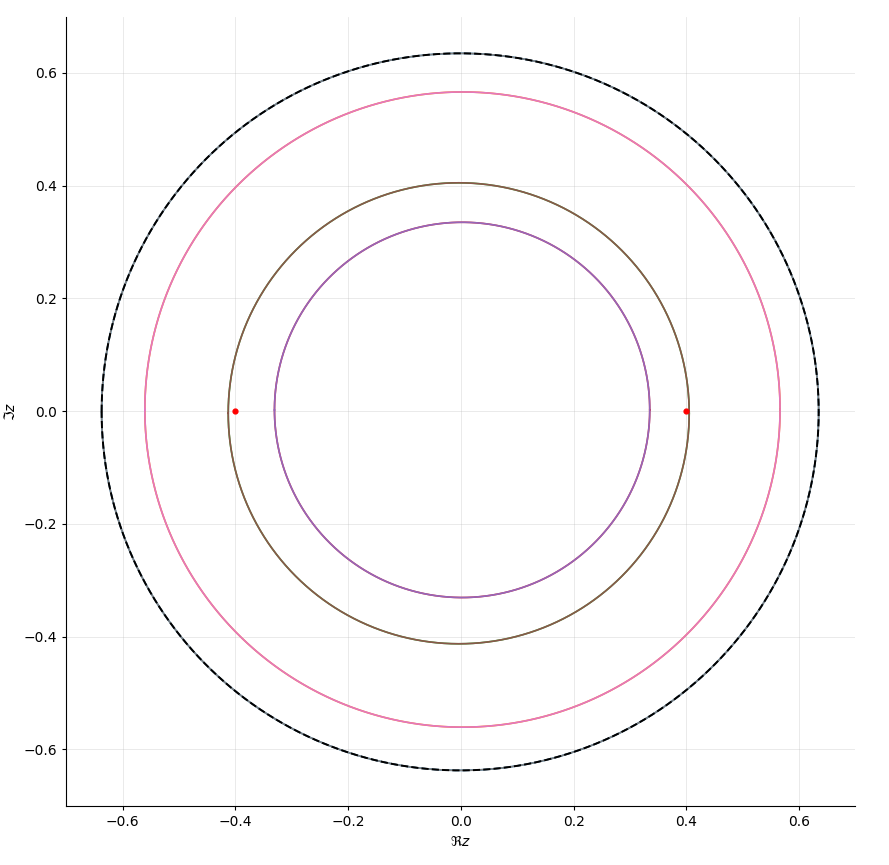}
    \caption{The Kippenhahn curve of a partial isometry from Example~\ref{ex:P-not-Q}. Red dots are placed at $(-0.4, 0) $ and $(0.4, 0)$, to highlight that the green branch is not a circle.}
    \label{fig:noncircA_circC}
\end{figure}
\begin{rem}
    In fact one may check that $A_6$ is the first obstruction for $\mathcal{A}(C)$ to have Circularity property, meaning that $A_i = 0$ for $i \leq 5$ provided $C$ has Circularity property. 
\end{rem}

The bound in Theorem~\ref{thm:Q-implies-P-low-dim} is sharp: there is a size-$5$ nilpotent contraction $C$ such that
$$
Q_C(x,y)\in\C[xy]
\qquad\text{but}\qquad
P_C(x,y)\notin\C[xy].
$$
In particular, Corollary~\ref{cor:A-to-C-low-dim} fails in size $5$, yet the example we found has a rather involved construction.
\begin{prop}\label{ex:rotationa_A_nonrot_C}

Set
$$
C_\alpha=
\begin{bmatrix}
0&\frac12&s&c&0\\
0&0&\frac12&f&-\frac54c\\
0&0&0&\frac12&\alpha s\\
0&0&0&0&\frac25\\
0&0&0&0&0
\end{bmatrix}.
$$

    Let $\Phi$ be the quintic
$$
\Phi(k)=24182784k^5-927166375k^4+274015368k^3+10740240k^2-318470912k-47232000.
$$
Choose any real root $\alpha\in(-0.55,-0.54)$ and set
$$
s^2=\frac{36(4\alpha+5)}{1025\alpha(5\alpha-4)},
\qquad
c=\frac{32(4\alpha+5)}{205(5\alpha-4)},
\qquad
f=-\frac{5(4\alpha+5)}{41}s.
$$
Then $\mathcal{A}(C_\alpha)$ has Circularity property, while $C_\alpha$ has not.
\end{prop}
\begin{proof}
We show that $Q_{C_\alpha}(x,y)\in\C[xy]$, whereas $P_{C_\alpha}(x,y)\notin\C[xy]$.
 Explicit computations show, provided $C_\alpha$ has the above form, that the only remaining coefficient of $Q_{C_\alpha}$ that does not correspond to the power of $xy$ is
\[
q_{45}=-\frac{s\,\Phi(\alpha)}{1130304400\,\alpha(25\alpha^2-40\alpha+16)},
\]
it corresponds to $x^4y^5$. So $Q_{C_\alpha}(x,y)\in\C[xy]$ whenever $\Phi(\alpha)=0$.
At the same time the only surviving coefficient of $P_{C_\alpha}$ that does not correspond to the power of $xy$ is
\[
p_{23}=-\frac{s\,\Psi(\alpha)}{282576100\,\alpha(25\alpha^2-40\alpha+16)},
\]
where
\[
\Psi(k)=6045696k^5-285636125k^4+127086692k^3+12334000k^2-97261504k-11808000.
\]
Moreover,
\[
\gcd(\Phi,\Psi)=1,
\]
so for any root $\alpha$ of $\Phi$ one has $\Psi(\alpha)\neq 0$, hence $p_{23}\neq 0$ and $P_{C_\alpha}(x,y)\notin\C[xy]$. Finally, there is a root of $\Phi(k)$ on the specified interval by the intermediate value theorem, and with this choice of $\alpha$ the operator $C_\alpha$ is indeed a contraction.
\end{proof}

 We conclude this section with an example of a $5$-by-$5$ contraction whose associated partial isometry has circular numerical range, yet does not have Circularity property.

\begin{prop}\label{prop:clean-counterexample}

Let \(s\in(0,\tfrac14)\) be the smaller root of
\begin{equation}\label{eq:def-s-clean}
g(s):=
4096(7+\sqrt{145})\,s^2-32(1403+125\sqrt{145})\,s+(9091+757\sqrt{145})=0.
\end{equation}
Set
\[
p:=\sqrt{s},
\qquad
C:=
\begin{bmatrix}
0&\frac12&0&0&0\\
0&0&\frac14&p&0\\
0&0&0&\frac14&-p\\
0&0&0&0&\frac12\\
0&0&0&0&0
\end{bmatrix},
\]
\[
D:=I_5-C^*C,
\qquad
B:=D^{1/2},
\qquad
A:=
\begin{bmatrix}
0&B\\
0&C
\end{bmatrix}.
\]
Then:
\begin{enumerate}[label=\textup{(\roman*)}]

    \item \(A\) is unitarily irreducible;
    \item \(W(A)\) is a disc;
    \item \(Q_C(x,y)\notin\mathbb C[xy]\).
\end{enumerate}
That is, \(A\) has circular numerical range but does not have Circularity property.
\end{prop}

\begin{proof}

Since
\[
g(0)=9091+757\sqrt{145}>0
\]
and
\[
g\!\left(\frac14\right)=-341+13\sqrt{145}<0,
\]
the smaller root \(s\) of \eqref{eq:def-s-clean} indeed belongs to \((0,\tfrac14)\). In particular
\[
p=\sqrt{s}>0.
\]

First we show that $A$ is unitarily irreducible using Proposition~\ref{prop:irr}. 

A direct computation gives
\[
D=
\begin{bmatrix}
1&0&0&0&0\\
0&\frac34&0&0&0\\
0&0&\frac{15}{16}&-\frac p4&0\\
0&0&-\frac p4&\frac{15}{16}-p^2&\frac p4\\
0&0&0&\frac p4&\frac34-p^2
\end{bmatrix}.
\]
Its leading principal minors are
\[
1,\qquad \frac34,\qquad \frac{45}{64},
\qquad
\frac{3(225-256s)}{1024},
\qquad
\frac{3(1024s^2-1728s+675)}{4096}.
\]
Since \(0<s<\frac14\), we have
\[
225-256s>225-64=161>0,
\]
and
\[
1024s^2-1728s+675>1024\cdot 0-\frac{1728}{4}+675=243>0.
\]
Hence all leading principal minors are positive, so \(D\succ0\) by Sylvester's criterion. In particular \(B=D^{1/2}\) is invertible.
It remains to show that $C$ is unitarily irreducible. Indeed, $C^4 \neq 0$, while $C^5 = 0$. If $C$ were (not necessarily unitarily) similar to a direct sum of two smaller matrices, then $C^4$ would be equal to zero.

Now we show that $W(A)$ is a disc.
As before, for \(\theta\in\mathbb R\), set
\[
H_A(\theta):= \frac{1}{2} (e^{-i\theta}A+e^{i\theta}A^*).
\]
Then
\[
2\omega I_{10}-2H_A(\theta)
=
\begin{bmatrix}
2\omega I_5&-e^{-i\theta}B\\
-e^{i\theta}B&2\omega I_5-e^{-i\theta}C-e^{i\theta}C^*
\end{bmatrix},
\]
where 
\[
\omega:=\frac18\sqrt{19+\sqrt{145}}.
\]
Since \(2\omega I_5\) is positive definite, Schur complement shows that
\[
2\omega I_{10}-2H_A(\theta) \succeq 0
\iff
M(\theta)\succeq 0,
\]
where
\[
M(\theta):=
(4\omega^2-1)I_5 + C^*C -2\omega\bigl(e^{-i\theta}C+e^{i\theta}C^*\bigr).
\]

We now compute the leading principal minors \(\Delta_k(\theta)\) of \(M(\theta)\). Since
\[
\omega^2=\frac{19+\sqrt{145}}{64},
\]
a direct computation gives
\[
\Delta_1=\frac{3+\sqrt{145}}{16},
\qquad
\Delta_2=\frac{45+3\sqrt{145}}{128},
\qquad
\Delta_3=\frac{257+23\sqrt{145}}{1024},
\]
hence \(\Delta_1,\Delta_2,\Delta_3>0\).

For the fourth leading principal minor one obtains
\[
\Delta_4(\theta)=
\frac{1859+149\sqrt{145}}{8192}
-\frac{83+5\sqrt{145}}{128}\,s
-\frac{(7+\sqrt{145})\sqrt{19+\sqrt{145}}}{512}\,\sqrt{s}\cos\theta.
\]
Since \(0<s<\frac14\), we have \(\sqrt{s}<\frac12\), and therefore
\begin{align*}
\Delta_4(\theta)
&\ge
\frac{1859+149\sqrt{145}}{8192}
-\frac{83+5\sqrt{145}}{128}\cdot\frac14
-\frac{(7+\sqrt{145})\sqrt{19+\sqrt{145}}}{512}\cdot\frac12\\
&=
\frac{531+69\sqrt{145}}{8192}
-\frac{(7+\sqrt{145})\sqrt{19+\sqrt{145}}}{1024}.
\end{align*}
To see that this lower bound is positive, it is enough to verify
\[
531+69\sqrt{145}
>
8(7+\sqrt{145})\sqrt{19+\sqrt{145}}.
\]
Squaring both sides gives
\[
(531+69\sqrt{145})^2
-64(7+\sqrt{145})^2(19+\sqrt{145})
=
606482+43838\sqrt{145}>0.
\]
Hence \(\Delta_4(\theta)>0\) for all \(\theta\in\mathbb R\).

Finally, the full determinant is independent of \(\theta\), and equals
\[
\Delta_5(\theta)
=
\frac{g(s)}{65536}=0
\]
where the latter equality holds by the defining equation \eqref{eq:def-s-clean}.

Thus, for every \(\theta\),
\[
\Delta_1(\theta)>0,\qquad
\Delta_2(\theta)>0,\qquad
\Delta_3(\theta)>0,\qquad
\Delta_4(\theta)>0,\qquad
\Delta_5(\theta)=0.
\]
That immediately implies 
\[
W(A)=\overline{\mathbb D}_{\omega}.
\]

Finally, a direct expansion gives
\begin{align*}
Q_C(x,y)
&=
1-5xy+\left(\frac{75}{8}-2p^2\right)x^2y^2
+\frac p{16}(x^3y^2+x^2y^3)\\
&\quad
+\left(\frac{91}{16}p^2-\frac{2089}{256}\right)x^3y^3
-\frac{19p}{128}(x^4y^3+x^3y^4)
+\left(p^4-\frac{313}{64}p^2+\frac{1683}{512}\right)x^4y^4\\
&\quad
+\frac{27p}{512}(x^5y^4+x^4y^5)
+\left(-\frac34p^4+\frac{81}{64}p^2-\frac{2025}{4096}\right)x^5y^5.
\end{align*}
In particular, the coefficient of \(x^3y^2\) is
\[
q_{32}=\frac p{16}\neq 0.
\]
Therefore \(Q_C(x,y)\) cannot be a polynomial in \(xy\) alone.

By Proposition~\ref{prop:bridge}, this means that \(A\) does not have Circularity property, even though \(W(A)\) is a disk.
\end{proof}

\begin{rem}
Numerically,
\[
s\approx 0.2466,\qquad
p\approx 0.4965,\qquad
\omega\approx 0.6964.
\]
 We plot the Kippenhahn curve of $\mathcal{A}(C)$ using numerical approximation, see Figure~\ref{fig:no_circ_prop_ex}

\end{rem}

\begin{figure}[H]
   
    \includegraphics[width=0.9\linewidth]{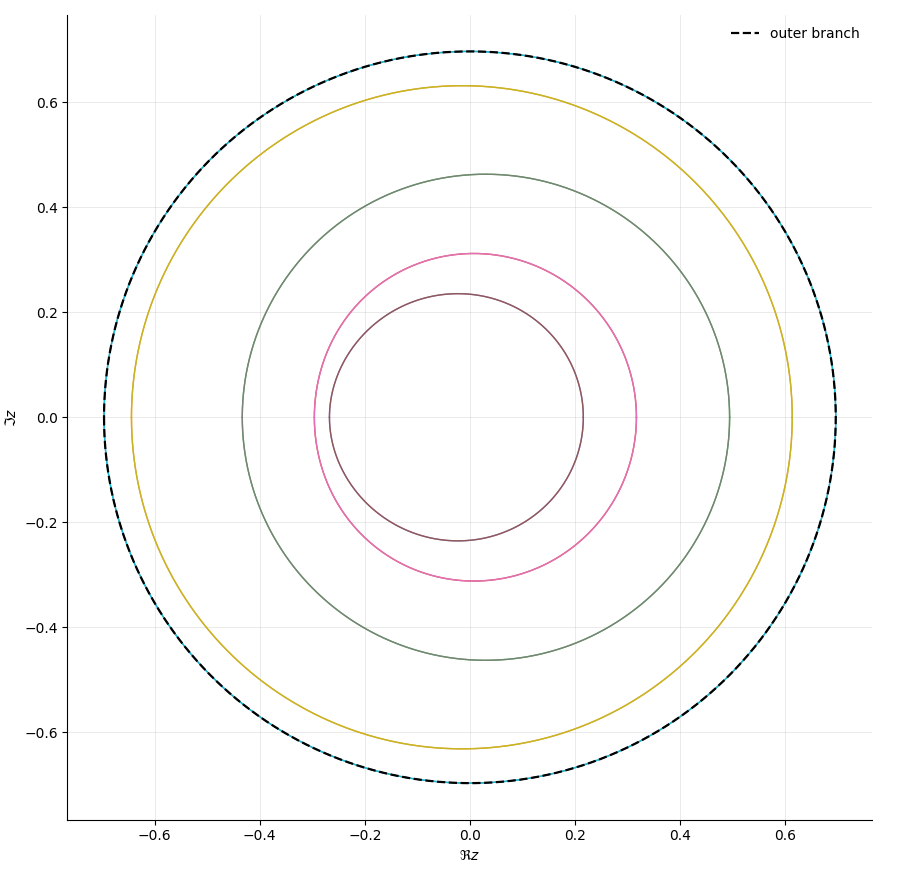}
    \caption{The Kippenhahn curve of a nilpotent partial isometry with circular numerical range without Circularity property. The outer branch is a circle, hence the numerical range is a disc.}
    \label{fig:no_circ_prop_ex}
\end{figure}
\begin{rem}
    
We did not find explicit examples of unitarily irreducible nilpotent partial isometries of sizes $7$, $8$ and $9$ with circular numerical ranges, so that they do not have Circularity property. Despite that, our numerical experiments suggest that there are such examples.

\end{rem}

\section{Computable criteria for rotational invariance}\label{sec:overlap}
The approaches we used to construct our examples in previous sections required either computing various traces of matrix polynomials or determinants of matrix polynomials. Say, proving that the operator from Example~\ref{ex:tower-counterexample} is not rotationally invariant required us to find an imbalanced word in $C$ and $C^*$ with non-zero trace. Such a search may become impractical even for operators of moderate size.
This section addresses the practical problem of efficiently detecting that a circular matrix is \emph{not} rotationally invariant.
 We recall the following structural result due to Li and Tsing \cite{LiTsing}:

 \begin{thm}\label{thm:rot_inv_struct}
     A unitarily irreducible $n$-by-$n$ matrix $X$ is rotationally invariant if and only if it is unitarily similar to the matrix of the form
     $$\begin{bmatrix}
         0 & X_1 & & \\
           &  0  & \ddots & \\
           &     & \ddots & X_{k-1} \\
           & &  & 0 \\
     \end{bmatrix}$$
  \end{thm}
That is, there is an orthogonal decomposition 
$$
\C^n=H_1\oplus\cdots\oplus H_k.
$$
such that $X(H_i) \subset H_{i-1}$ for $i \geq 2$ and $X(H_1) = \left\{0\right\}$. 
In what follows we will use the proof of sufficiency, which we recall here.
\begin{prop}\label{prop:gauge}
Assume that $X$ is given as just described.

Define
$$
U_\theta:=\diag\bigl(\Id_{H_1},e^{i\theta}\Id_{H_2},\dots,e^{i(k-1)\theta}\Id_{H_k}\bigr).
$$
Then
$$
U_\theta^*XU_\theta=e^{i\theta}X,
\qquad
H_X(\theta)=U_\theta^*H_X(0)U_\theta.
$$
\end{prop}
Note that $H_X(\theta) = \re(e^{-i\theta}X)$ depends on $\theta$ analytically. It is well known that corresponding eigenvalues 
$$ \lambda_1(\theta), \, \lambda_2(\theta), \, \ldots, \, \lambda_n(\theta)$$ and unit eigenvectors 
$$y_1(\theta), \, y_2(\theta), \, \ldots, \, y_n(\theta)$$
may be chosen so that they depend analytically on $\theta$ as well. 
We refer the reader to \cite{GaSe12} for the details. In the case when $X$ is rotationally invariant, $H_X(\theta)$ are isospectral for every $\theta$. That is, $\lambda_i(\theta) \equiv \lambda_i(0)$. Until the end of this section we denote $\lambda_i = \lambda_i(\theta)$ and assume
$$\lambda_1 \geq \lambda_2 \geq \ldots \geq \lambda_n.$$

The following Theorem is largely inspired by work of Gau and Wu \cite{WangWu11}, where they apply the same logic to study weighted-shifts. What follows is an immediate generalization to block-shifts.

\begin{thm}\label{thm:gauwu_poly}
Assume that $X$ is rotationally invariant. Let $\lambda_j$ be an eigenvalue of $\re(\exp(-i\theta) X)$ and let $y^{(j)}(\theta)$ be the corresponding unit eigenvector, chosen to be analytic in $\theta$. Assume in addition that $\lambda_j$ is simple. Then there is a polynomial
$$
p_j(z)=\sum_{m=1}^k c_m^{(j)}z^{m-1},
\qquad c_m^{(j)}\ge 0,
\qquad \sum_{m=1}^k c_m^{(j)}=1,
$$
and a real-analytic function $\psi$ such that 
$$
\inn{y^{(j)}(\theta)}{y^{(j)}(\phi)}
=
p_j\bigl(e^{i(\phi-\theta)}\bigr)e^{i(\psi(\phi)-\psi(\theta))}
\qquad (\theta,\phi\in\R).
$$
\end{thm}

\begin{proof}
Fix a simple eigenvalue branch of $H_X(\theta)$ and choose a unit eigenvector $x^{(j)}(0)$ at $\theta=0$. Write
$$
x^{(j)}(0)=x_1^{(j)}\oplus\cdots\oplus x_k^{(j)},
\qquad x_m^{(j)}\in H_m,
$$
and define the coefficients
$$
c_m^{(j)}:=\norm{x_m^{(j)}}^2,
\qquad m=1,\dots,k.
$$
Then $c_m^{(j)}\ge 0$ and $\sum_m c_m^{(j)}=1$.
Consider $x^{(j)}(\theta)=U_\theta^*x^{(j)}(0)$. By Proposition~\ref{prop:gauge}, this is an eigenvector of $H_X(\theta)$ that corresponds to $\lambda_j$ and depends on $\theta$ analytically. Then
\begin{align*}
\inn{x^{(j)}(\theta)}{x^{(j)}(\phi)}
&=\inn{x^{(j)}(0)}{U_{\phi-\theta}^*x^{(j)}(0)}\\
&=\sum_{m=1}^k e^{i(m-1)(\phi-\theta)}\norm{x_m^{(j)}}^2
=\sum_{m=1}^k c_m^{(j)}e^{i(m-1)(\phi-\theta)},
\end{align*}
which is the asserted formula with $\psi \equiv const$. Every other possible choice of $y^{(j)}$ is obtained from the one we just described by pointwise multiplication by an analytic function, that is $y^{(j)}(\theta) = e^{i\psi(\theta)}x^{(j)}(\theta)$, which gives the asserted formula.
\end{proof}

\begin{defi} \label{defi:simple_eig_branch}
    In the situation of Theorem~\ref{thm:gauwu_poly}, we call $y^{(j)}$ a simple eigenvector branch. 
\end{defi}

Thus, the scalar product of two eigenvectors corresponding to the same eigenvalue but to different angles is, up to multiplication by a unimodular scalar, is determined by the values of a certain polynomial.

To eliminate that unit complex number and simplify things a bit, one may take an absolute value.
In particular, the following holds.
\begin{cor}\label{cor:magnitude}
Assume that $X$ is rotationally invariant and let $y(\theta)$ be a simple eigenvector branch. Then
the quantity
$$
\abs{\inn{y(\theta)}{y(\theta + \delta)}}
$$
depends only on $\delta$.
\end{cor}
One may use it as a computationally inexpensive practical test for rotational invariance of an operator $X$ in the following way (of course, details may vary):
\begin{enumerate}
    \item Take a step $\delta$ and two different values $\theta_1$ and $\theta_2$.
    \item Find (not necessarily) largest eigenpairs for each of corresponding real parts $H_X(\theta_1), H_X(\theta_1 + \delta)$ and $H_X(\theta_2), H_X(\theta_2 + \delta)$. Denote obtained eigenvectors as $y_1(\theta_1), y_1(\theta_1+\delta)$ and $y_2(\theta_2), y_2(\theta_2+\delta)$, respectively.
    \item If obtained eigenvalues are different, then the operator $X$ does not have Circularity property, hence it is not rotationally invariant.
    \item If obtained eigenvalues are simple, check if $$\abs{\inn{y_1(\theta_1)}{y_1(\theta_1 + \delta)}} = \abs{\inn{y_2(\theta_2)}{y_2(\theta_2 + \delta)}}.$$ If not, then the operator $X$ is not rotationally invariant.
\end{enumerate}
We illustrate this approach by presenting a sequence of matrices in every even dimension that have Circularity property and yet are not rotationally invariant.

For $n\ge 4$ define
$$
r_n:=\left\lfloor\frac{n-2}{2}\right\rfloor,
\qquad
s_n:=n-1-r_n,
$$
and consider the family
\begin{equation}\label{eq:centered-family}
C_n(1,1)=J_n+E_{r_n,r_n+2}-E_{s_n,s_n+2}.
\end{equation}
The first few members are
$$
C_4(1,1)=
\begin{bmatrix}
0&1&1&0\\
0&0&1&-1\\
0&0&0&1\\
0&0&0&0
\end{bmatrix},
\qquad
C_5(1,1)=
\begin{bmatrix}
0&1&1&0&0\\
0&0&1&0&0\\
0&0&0&1&-1\\
0&0&0&0&1\\
0&0&0&0&0
\end{bmatrix}.
$$
\[
C_6(1,1)=
\begin{bmatrix}
0&1&0&0&0&0\\
0&0&1&1&0&0\\
0&0&0&1&-1&0\\
0&0&0&0&1&0\\
0&0&0&0&0&1\\
0&0&0&0&0&0
\end{bmatrix}.
\]

\begin{thm}\label{thm:circular-family}
For every $m\ge 2$, the matrix $C_{2m}(1,1)$ has Circularity property.
More precisely, 
$$\det\!\bigl(uI-H_{C_{2m}}(\theta) \bigr)
=
2^{-2m}
\bigl(U_m(u)-U_{m-1}(u)-U_{m-2}(u)\bigr)
\bigl(U_m(u)+U_{m-1}(u)-U_{m-2}(u)\bigr),$$

where \(U_k\) denotes the Chebyshev polynomial of the second kind.
\end{thm}

\begin{proof}
Let $M_\theta(u):= 2uI-H_{C_{2m}}(\theta)$
and 
$D_\theta:=\operatorname{diag}(1,e^{-i\theta},e^{-2i\theta},\dots,e^{-(2m-1)i\theta})$, then

$$
\det M_\theta(u)=\det D_\theta^*M_\theta(u)D_\theta,
$$
A direct computation shows that
\[
D_\theta^*M_\theta(u)D_\theta
=
T_{2m}(u)
-e^{-i\theta}E_{m-1,m+1}+e^{-i\theta}E_{m,m+2}
-e^{i\theta} E_{m+1,m-1}+e^{i\theta} E_{m+2,m},
\]
where \(T_{2m}(u)\) is the Toeplitz tridiagonal matrix
\[
T_{2m}(u)=
\begin{bmatrix}
2u & -1 \\
-1 & 2u & -1 \\
& \ddots & \ddots & \ddots \\
&& -1 & 2u & -1\\
&&& -1 & 2u
\end{bmatrix}.
\]

For \(r\ge 0\), it holds that
$$\det T_r(u) = U_r(u). $$

Now partition the index set \(\{1,\dots,2m\}\) into
\[
L=\{1,\dots,m-2\},\qquad
C=\{m-1,m,m+1,m+2\},\qquad
R=\{m+3,\dots,2m\}.
\]
Relative to this decomposition, \(N_\theta(u)\) has block form
\[
N_\theta(u)=
\begin{bmatrix}
T_{m-2}(u) & B_L & 0\\
B_L^* & K_\theta(u) & B_R^*\\
0 & B_R & T_{m-2}(u)
\end{bmatrix},
\]
where \[
B_L=-E_{m-2,\,1}\in M_{m-2,\,4},
\qquad
B_R=-E_{1,\,4}\in M_{m-2,\,4},
\]
\[
K_\theta(u)=
\begin{bmatrix}
2u & -1 & -e^{-i\theta} & 0\\
-1 & 2u & -1 & e^{-i\theta}\\
-e^{i\theta}& -1 & 2u & -1\\
0 & e^{i\theta} & -1 & 2u
\end{bmatrix}.
\]

Assume first that \(U_{m-2}(u)\neq 0\). Taking the Schur complement of the two tail blocks \(T_{m-2}(u)\), we obtain
\[
\det N_\theta(u)=U_{m-2}(u)^2\det \widetilde K_\theta(u),
\]
where
\[
\widetilde K_\theta(u)=
\begin{bmatrix}
a & -1 & -e^{-i\theta} & 0\\
-1 & 2u & -1 & e^{-i\theta}\\
-e^{i\theta} & -1 & 2u & -1\\
0 & e^{i\theta} & -1 & a
\end{bmatrix},
\qquad
a=\frac{U_{m-1}(u)}{U_{m-2}(u)}.
\]

A direct \(4\times 4\) determinant computation gives
\[
\det \widetilde K_\theta(u)
=
4a^2u^2-a^2-8au+4.
\]
 Thus
\[
\det N_\theta(u)
=
U_{m-2}(u)^2\bigl(4a^2u^2-a^2-8au+4\bigr).
\]
Substituting \(a=U_{m-1}/U_{m-2}\), we get
\[
\det N_\theta(u)
=
4u^2U_{m-1}(u)^2-U_{m-1}(u)^2-8uU_{m-1}(u)U_{m-2}(u)+4U_{m-2}(u)^2.
\]
Finally, using
\[
U_m(u)=2uU_{m-1}(u)-U_{m-2}(u),
\]
we factor this as
\[
\det N_\theta(u)
=
\bigl(U_m(u)-U_{m-1}(u)-U_{m-2}(u)\bigr)
\bigl(U_m(u)+U_{m-1}(u)-U_{m-2}(u)\bigr).
\]
\end{proof}

We compute largest eigenpairs of $H_{C_{2m}}(\theta)$ for three different values of $\theta$, then compute their scalar product and apply Corollary~\ref{cor:magnitude}.

\begin{prop}\label{prop:three-largest-eigenpairs-even-compact}
Let
\[
C=C_{2m}(1,1)=J_{2m}+E_{m-1,m+1}-E_{m,m+2},
\qquad m\ge 2.
\]
Define
\[
P_m(x):=U_m(x)-U_{m-1}(x)-U_{m-2}(x).
\]
Let \(\lambda_m\) be the largest root of \(P_m\). Then \(\lambda_m\) is the largest eigenvalue of
\[
H_C(0)=\Re C,\qquad H_C(-\pi/4) =\Re(e^{i\pi/4}C),\qquad H_C(-\pi/2)=\Re(iC)=-\Im C.
\]

Moreover, corresponding eigenvectors may be chosen as follows.

\smallskip

\noindent
\textup{(i) For \(H_C(0)\)} the eigenvector is given by \(v_0=(v_0(1), v_0(2), \ldots, v_0(2m))\), where
\[
v_0(k)=
\begin{cases}
U_{k-1}(\lambda_m), & 1\le k\le m-1,\\[1mm]
\dfrac12\,U_{m-1}(\lambda_m), & k=m,m+1,\\[1mm]
0, & m+2\le k\le 2m.
\end{cases}
\]

\smallskip

\noindent
\textup{(ii) For \(H_C(-\pi/4)\):} let
\[
\eta:=e^{-i\pi/4},\qquad
\alpha:=\frac{1}{1+\eta},\qquad
q:=\frac{1-\bar\eta}{1+\eta}=(\sqrt2-1)e^{-i\pi/4}.
\]
The eigenvector is given by \(v_{\pi/4}=(v_{\pi/4}(1), v_{\pi/4}(2), \ldots, v_{\pi/4}(2m))\), where
\[
v_{\pi/4}(k)=
\begin{cases}
\eta^{\,k-1}U_{k-1}(\lambda_m), & 1\le k\le m-1,\\[1mm]
\eta^{\,k-1}\alpha\,U_{m-1}(\lambda_m), & k=m,m+1,\\[1mm]
\eta^{\,k-1}q\,U_{2m-k}(\lambda_m), & m+2\le k\le 2m.
\end{cases}
\]

\smallskip

\noindent
\textup{(iii) For \(H_C(-\pi/2)\):} let
\[
\beta:=\frac{1+i}{2}.
\]
The eigenvector is given by  \(v_{\pi/2}=(v_{\pi/2}(1), v_{\pi/2}(2), \ldots, v_{\pi/2}(2m))\), where
\[
v_{\pi/2}(k)=
\begin{cases}
(-i)^{\,k-1}U_{k-1}(\lambda_m), & 1\le k\le m-1,\\[1mm]
(-i)^{\,k-1}\beta\,U_{m-1}(\lambda_m), & k=m,m+1,\\[1mm]
(-i)^{\,k-1}U_{2m-k}(\lambda_m), & m+2\le k\le 2m.
\end{cases}
\]
\end{prop}

\begin{proof}[Proof sketch]
For each of the three matrices the eigenvalue equation reduces to a Hermitian matrix whose first super- and subdiagonal entries are all equal to \(1/2\), with only two modified centered entries. Away from the center, the eigenvalue equation is therefore
\[
2\lambda x_k=x_{k-1}+x_{k+1},
\]
so the tails are governed by the recurrence for \(U_k(\lambda)\).

The displayed vectors are obtained by solving the finite matching problem at the central rows \(m-1,m,m+1,m+2\). In all three cases the central compatibility condition reduces to
\[
(2\lambda-1)U_{m-1}(\lambda)=2U_{m-2}(\lambda),
\]
equivalently
\[
P_m(\lambda)=U_m(\lambda)-U_{m-1}(\lambda)-U_{m-2}(\lambda)=0.
\]

For \(H_C(0)\), one has
\[
\det(xI-H_C(0))=2^{-2m}P_m(x)Q_m(x),
\qquad
Q_m(x):=U_m(x)+U_{m-1}(x)-U_{m-2}(x),
\]
and \(Q_m(x)>0\) for \(x\ge 1\), whereas \(P_m(1)=2-m\le 0\) and \(P_m(x)\to+\infty\) as \(x\to+\infty\). Hence the largest eigenvalue of \(H_C(0)\) is the largest root \(\lambda_m\) of \(P_m\). Since \(C\) has the Circularity property, the matrices \(H_C(\theta)\) are isospectral, so the same \(\lambda_m\) is the largest eigenvalue of \(H_C(-\pi/4)\) and \(H_C(-\pi/2)\) as well.
\end{proof}

\begin{thm}\label{thm:even-overlap-comparison-compact}
Let \(C=C_{2m}(1,1)\), and let \(v_0,v_{\pi/4},v_{\pi/2}\) be the vectors from Proposition~\ref{prop:three-largest-eigenpairs-even-compact}. Set $\widehat v_\theta:=\frac{v_\theta}{\|v_\theta\|},
\, \theta\in\{0,\pi/4,\pi/2\}$.
Then
\[
\left|\inn{\widehat v_0}{\widehat v_{\pi/4}}\right|
>
\left|\inn{\widehat v_{\pi/4}}{\widehat v_{\pi/2}}\right|.
\]
In particular, \(C_{2m}(1,1)\) is not rotationally invariant.
\end{thm}

\begin{proof}
Define
\[
M_m:=\frac12+\sum_{k=1}^{m-1}
\frac{U_{m-1-k}(\lambda_m)^2}{U_{m-1}(\lambda_m)^2},
\]
\[
S_m:=\sum_{k=1}^{m-1}
e^{ik\pi/4}\,
\frac{U_{m-1-k}(\lambda_m)^2}{U_{m-1}(\lambda_m)^2}.
\]

A direct substitution of the explicit coordinates from Proposition~\ref{prop:three-largest-eigenpairs-even-compact} gives the norm identities
\[
\|v_0\|^2=U_{m-1}(\lambda_m)^2M_m,
\qquad
\|v_{\pi/4}\|^2=2(2-\sqrt2)\,U_{m-1}(\lambda_m)^2M_m,
\]
\[
\|v_{\pi/2}\|^2=2\,U_{m-1}(\lambda_m)^2M_m,
\]
as well as the scalar product formulas
\[
\inn{v_0}{v_{\pi/4}}
=
e^{-i(m-1)\pi/4}U_{m-1}(\lambda_m)^2\left(\frac12+S_m\right),
\]
and
\[
\inn{v_{\pi/4}}{v_{\pi/2}}
=
e^{-i(m-1)\pi/4}U_{m-1}(\lambda_m)^2
\left(
\frac1{\sqrt2}+S_m+(\sqrt2-1)\overline{S_m}
\right).
\]
This yields
\[
\left|\inn{\widehat v_0}{\widehat v_{\pi/4}}\right|
=
\frac{\left|\frac12+S_m\right|}
{\sqrt{2(2-\sqrt2)}\,M_m},
\]
and
\[
\left|\inn{\widehat v_{\pi/4}}{\widehat v_{\pi/2}}\right|
=
\frac{\left|\frac1{\sqrt2}+S_m+(\sqrt2-1)\overline{S_m}\right|}
{2\sqrt{2-\sqrt2}\,M_m}.
\]

Write \(S_m=x_m+iy_m\). A direct algebraic simplification gives
\[
2\left|\frac12+S_m\right|^2-
\left|\frac1{\sqrt2}+S_m+(\sqrt2-1)\overline{S_m}\right|^2
=
4(\sqrt2-1)\,y_m^2.
\]
Thus it is enough to prove that \(\Im S_m=y_m>0\).

But
\[
\Im S_m=
\sum_{k=1}^{m-1}
\sin\frac{k\pi}{4}\,
\frac{U_{m-1-k}(\lambda_m)^2}{U_{m-1}(\lambda_m)^2}.
\]
Since \(\lambda_m>1\), the sequence \(U_r(\lambda_m)\) is strictly increasing in \(r\), so the coefficients
\[
a_k:=\frac{U_{m-1-k}(\lambda_m)^2}{U_{m-1}(\lambda_m)^2}
\]
are positive and strictly decreasing in \(k\). Grouping the above sine sum into blocks of length \(8\), and using
\[
\sin\frac{(k+4)\pi}{4}=-\sin\frac{k\pi}{4},
\]
one sees that each full block contributes a strictly positive amount, while any terminal incomplete block contributes a nonnegative amount. Hence \(\Im S_m>0\), and the desired strict inequality follows.

If \(C_{2m}(1,1)\) were rotationally invariant, Corollary~\ref{cor:magnitude} would force these two same-step magnitudes to coincide. This contradiction shows that \(C_{2m}(1,1)\) is not rotationally invariant.
\end{proof}

\bibliographystyle{amsplain}
\bibliography{master}

\end{document}